\documentclass[11pt]{article}
\usepackage{amsmath,amssymb,amsthm,graphicx}
\usepackage[mathscr]{eucal}
\usepackage{color}
\usepackage[english]{babel}
\usepackage[latin1]{inputenc}

\usepackage[usenames,dvipsnames]{xcolor} 
\usepackage{graphicx} 
\usepackage{draftcopy} 
\draftcopyName{WATERMARK}{100} 
\usepackage{watermark} 
\definecolor{verylightblue}{rgb}{.855,.89,1.0}
\definecolor{lightbluegray}{rgb}{.788,.847,0.887}
\def\watermarktext{\put(-10,-350){
\rotatebox{
90}{\Huge\bf\color{lightbluegray}
{\scriptsize (for arXiv v1.28092015)}
}}}
\watermark{%
\watermarktext 
} 

\newtheorem{theorem}{Theorem}[section]
\newtheorem{proposition}[theorem]{Proposition}
\newtheorem{corollary}[theorem]{Corollary}
\newtheorem{lemma}[theorem]{Lemma}
\newtheorem{note}[theorem]{Note}
\newtheorem*{mainlemma}{Main Lemma}

\def\im{\mathop{\mathrm{Im}}\nolimits}
\def\rank{\mathop{\mathrm{rank}}\nolimits}
\def\N{\mathbb N}
\def\Z{\mathbb Z}
\def\R{\mathbb R}
\def\Q{\mathbb Q}
\def\O{\mathcal{O}}
\def\T{\mathcal{T}}
\def\S{\mathcal{S}}
\def\E{\mathcal{E}}
\newcommand{\lastpage}{\addresss} 
\newcommand{\addresss}{\small \sf  

\noindent{\sc Ilinka Dimitrova}, 
Faculty of Mathematics and Natural Science, 
South-West University "Neofit Rilski", 
2700 Blagoevgrad, 
Bulgaria;  
email: ilinka\_dimitrova@swu.bg

\medskip 

\noindent{\sc V\'\i tor H. Fernandes}, 
CMA, Departamento de Matem\'atica, 
Faculdade de Ci\^encias e Tecnologia, 
Universidade NOVA de Lisboa, 
Monte da Caparica, 
2829-516 Caparica, 
Portugal; 
e-mail: vhf@fct.unl.pt

\medskip

\noindent{\sc J\"{o}rg Koppitz}, 
Institute of Mathematics, 
University of Potsdam, 
14469 Potsdam, 
Germany; 
email: koppitz@uni-potsdam.de 
}

\title{A note on generators of the endomorphism semigroup of an infinite countable chain}

\author{Ilinka Dimitrova, V\'\i tor H. Fernandes\footnote{This work was developed within the 
FCT Project UID/MAT/00297/2013 of CMA and of Departamento de Matem\'atica da Faculdade de Ci\^encias e Tecnologia da Universidade Nova de Lisboa.}~ 
and J\"{o}rg Koppitz}

\begin{document}

\maketitle

\begin{abstract}
In this note, we consider the semigroup $\O(X)$ of all order endomorphisms of an infinite chain $X$ and the subset $J$ of $\O(X)$ 
of all transformations $\alpha$ such that $|\im(\alpha)|=|X|$. 
For an infinite countable chain $X$, we give a necessary and sufficient condition on $X$ for $\O(X) = \langle J \rangle$ to hold. 
We also present a sufficient condition on $X$ for 
$\O(X) = \langle J \rangle$ to hold, for an arbitrary infinite chain $X$. 
\end{abstract}

\medskip

\noindent{\small 2010 \it Mathematics subject classification: \rm 20M20, 20M10.}

\noindent{\small\it Keywords: \rm infinite chain, endomorphism semigroup, generators, relative rank.}

\section*{Introduction}

The \textit{rank} of a semigroup $S$ is the minimum cardinality of a 
generating set of $S$. For a countable semigroup $S$, in particular, for a finitely generated semigroup $S$, 
determining the rank of $S$ is a natural question. 
Contrariwise, for an uncountable semigroup $S$, this concept has no interest, since the rank of
$S$ is always $|S|$. This last fact leads to the following notion. 
For a subset $A$ of a semigroup $S$, the \textit{relative rank} of $S$ modulo $A$ is the minimum cardinality of a 
subset $B$ of $S$ such that $\langle A \cup B \rangle = S$. This cardinal is denoted by $\rank(S:A)$.
It follows immediately from the definition that $\rank(S:A) = \rank(S: \langle A \rangle )$ and that $\rank(S:A) = 0$ if and
only if $A$ is a generating set of $S$.

The notion of relative rank was introduced by Ru\v skuc in \cite{Ruskuc:1994}, who proved  that the rank of a finite Rees matrix semigroup
$\mathcal{M}[G; I,\Lambda; P]$, with the sandwich matrix $P$ in normal form, is equal to
$\textrm{max}\{|I|,|\Lambda|, \rank (G : H)\}$, where $H$ is the subgroup of $G$ generated by the
entries of $P$.
In \cite{Howie&Ruskuc&Higgins:1998}, Howie et al. considered the relative ranks of the full transformation
semigroup $\T(X)$ on $X$, where $X$ is an infinite set, modulo some distinguished subsets of $\T(X)$. 
They showed that $\rank (\T(X) : \S(X)) = 2$, $\rank (\T(X) : \E(X)) = 2$ and $\rank (\T(X) : J) = 0$, where $\S(X)$ is the symmetric group on $X$, $\E(X)$ is the set of all idempotent transformations on $X$ and $J$ is the top $\mathcal{J}$-class of $\T(X)$, i.e. 
$J=\{\alpha\in\T(X)\mid |\im(\alpha)|=|X|\}$. 

\smallskip 

Throughout this paper, we will represent a chain only by its support set and, as usual, 
its order by the symbol $\le$. Let $X$ be a chain.  
A transformation $\alpha$ of $X$ is said to be \textit{order-preserving} or an (order) \textit{endomorphism} of $X$ if $x\le y$ implies $x\alpha\le y\alpha$, for all $x,y\in X$. We denote by $\O(X)$ the subsemigroup of $\T(X)$ of all (order) endomorphisms of $X$. 

\smallskip 

For a finite chain $X$, it is well known, and clear, that $\O(X)$ is a regular semigroup. 
The problem for an infinite chain $X$ is much more involved. 
Nevertheless, more generally, a characterization of those posets $P$ for which the semigroup of all endomorphisms of $P$ is regular was done by A\v\i zen\v stat in 1968 
\cite{Aizenstat:1968} and, independently, by Adams and Gould in 1989 \cite{Adams&Gould:1989}.  

Let $X$ be an infinite chain.  
A useful regularity criterion for the elements of $\O(X)$ was proved  in \cite{Mora&Kemprasit:2010} by Mora and Kemprasit, who 
deduced several previous known results based on it:  for instance, that $\O(\Z)$ is regular while $\O(\Q)$ and $\O(\R)$ are not regular, 
by considering their usual orders. 
In \cite{Fernandes&Honyam&Quinteiro&Singha:2014}, 
Fernandes et al. described the largest regular subsemigroup of $\O(X)$ and also Green's relations on $\O(X)$. 
The relative rank of $\T(X)$ modulo the subsemigroup $\O(X)$ was considered by Higgins et al. in \cite{Higgins&Mitchell&Ruskuc:2003}.
They showed that $\rank(\T(X) : \O(X)) = 1$, when $X$ is an arbitrary countable chain or an arbitrary well-ordered set, 
while $\rank(\T(\R) : \O(\R))$ is uncountable, by considering the usual order of $\R$. 

\smallskip 

For a fixed chain $X$, consider the following two subsets of the semigroup $\O(X)$:  
$$
J=\{\alpha\in\O(X)\mid |\im(\alpha)|=|X|\}
\quad\text{and}\quad 
J_f=\{\alpha\in\O(X)\mid |\im(\alpha)|<\aleph_0\}.
$$
Notice that $J_f$ is clearly an ideal of $\O(X)$. On the other hand, unlike the analogous set for $\T(X)$, $J$ is not necessarily a 
$\mathcal{J}$-class of $\O(X)$ (see \cite{Fernandes&Honyam&Quinteiro&Singha:2014}). 

In this note we study the relative rank of the semigroup $\O(X)$ modulo $J$. 
For an infinite countable chain $X$, we give a necessary and sufficient condition on $X$ for $\O(X) = \langle J \rangle$ to hold 
(notice that, for a finite $X$, $\O(X)=\langle J\rangle$ if and only if $|X|=1$). We also present a sufficient condition on $X$ for 
$\O(X) = \langle J \rangle$ to hold, for an arbitrary infinite chain $X$. 

\medskip 

For general background on Semigroup Theory, 
we refer the reader to Howie's book \cite{Howie:1995}.

\section{Main results}

Let $X$ be an infinite chain. Let $x\in X$ and define 
$$
(x]=\{y\in X\mid y\le x\}\quad\text{and}\quad [x)=\{y\in X\mid x\le y\} 
$$
(i.e the left and right order ideals generated by $x$). 
Define also
$$
X^0=\{x\in X\mid |(x]|=|X|=|[x)|\},
$$
$$
X^-=\{x\in X\mid |(x]|<|X|\}
$$
and
$$
X^+=\{x\in X\mid |[x)|<|X|\}.
$$
Notice that, since $X$ is an infinite set, if $x\in X^-$ (respectively, $x\in X^+$) then 
$|X|=|[x)|$ (respectively, $|X|=|(x]|$). Hence $X$ is a disjoint union of $X^-$, $X^0$ and $X^+$. 

\smallskip 

Let us consider the sets $\N$, $\Z_-=\Z\setminus(\N\cup\{0\})$, $\Z$, $\Q$ and $\R$, with their usual orders. Then, we have: 
\begin{enumerate}
\item $X^-=\N$, $X^0=\emptyset$ and $X^+=\emptyset$, if $X=\N$;
\item $X^-=\emptyset$, $X^0=\emptyset$ and $X^+=\Z_-$, if $X=\Z_-$;
\item $X^-=\emptyset$, $X^0=X$ and $X^+=\emptyset$, for $X\in\{\Z,\Q,\R\}$. 
\end{enumerate}

\smallskip 

Recall that, given two posets $P$ and $Q$ with disjoint supports, the \textit{ordinal sum} $P\oplus Q$ of $P$ and $Q$ (by this order) is the poset with support $P\cup Q$ such that $P$ and $Q$ are subposets of 
$P\oplus Q$ and $x<y$, for all $x\in P$ and $y\in Q$. This operation on posets is associative (but not commutative). For our purposes, it is convenient to admit empty posets.  
 
\smallskip 

Let $\overset{\mbox{\scalebox{0.4}{$\rightarrow\!\!\cdot\!\!\leftarrow$}}}{\Z}$ be the chain 
$\N\oplus\{0\}\oplus\Z_-$, with the usual orders on $\N$ and $\Z_-$. 
Then, being $X=\overset{\mbox{\scalebox{0.4}{$\rightarrow\!\!\cdot\!\!\leftarrow$}}}{\Z}$, 
we have $X^-=\N$, $X^0=\{0\}$ and $X^+=\Z_-$. 

\smallskip 

By considering $X^-$, $X^0$ and $X^+$ as subposets of $X$, we have the following decomposition of $X$: 

\begin{lemma}
Let $X$ be an infinite chain. Then $X=X^-\oplus X^0\oplus X^+$.
\end{lemma}
\begin{proof}
First, let $a\in X^-$ and $b\in X^0\cup X^+$. 
If $b\le a$ then $(b]\subseteq (a]$ and so $|X|=|(b]|\le |(a]|<|X|$, a contradiction. Then $a<b$. 
On the other hand, given $a\in X^-\cup X^0$ and $b\in X^+$, by a dual reasoning, we may show that $a<b$.  
This proves the lemma.
\end{proof}

\begin{note}\label{n1}
Let $X$ be an infinite chain and let $\alpha\in\O(X)$. If there exist $x^+\in X^+$ and $x^-\in X^-$ such that $x^+\alpha=x^-$ 
or $x^-\alpha=x^+$ then $\alpha\not\in J$. \em 

\smallskip 

In fact, suppose that $x^+\alpha=x^-$ (the other case can be treated dually). Then $\im(\alpha)\subseteq (x^-]\cup[x^+)\alpha$ and so 
$|\im(\alpha)|\le |(x^-]|+|[x^+)\alpha|\le |(x^-]|+|[x^+)|<|X|+|X|=|X|$, i.e. $\alpha\not\in J$, as required. 
\end{note}

\begin{note}\label{n2}
Let $X$ be an infinite chain such that $X^0=\emptyset$. Let $\alpha\in\O(X)$ be such that  $x^+\alpha=x^-$ 
or $x^-\alpha=x^+$, for some $x^+\in X^+$ and $x^-\in X^-$. Let $\alpha_1,\alpha_2,\ldots,\alpha_n\in\O(X)$ be such that 
$\alpha=\alpha_1\alpha_2\cdots\alpha_n$. Then $|\im(\alpha)|\le|\im(\alpha_i)|<|X|$ (and so $\alpha,\alpha_i\not\in J$), for some $i=1,2,\ldots,n$. \em 

\smallskip 

In fact, for the case $x^+\alpha=x^-$ (the other case is dual), let $i=\text{min}\{j\in\{1,\ldots,n\}\mid x^+\alpha_1\cdots\alpha_j\in X^-\}$. 
Then $x^+\alpha_1\cdots\alpha_{i-1}\in X^+$ (for $i=1$ the expression $x^+\alpha_1\cdots\alpha_{i-1}$ has the meaning of $x^+$) 
and $(x^+\alpha_1\cdots\alpha_{i-1})\alpha_i\in X^-$, since $X^0=\emptyset$. Hence, by Note \ref{n1}, $\alpha_i\not\in J$. 
On the other hand, from the equality $\alpha=\alpha_1\alpha_2\cdots\alpha_n$, it follows that $|\im(\alpha)|\le|\im(\alpha_j)|$ 
(indeed $\im(\alpha_1\cdots\alpha_j)\subseteq\im(\alpha_j)$, whence  
$\im(\alpha)=(\im(\alpha_1\cdots\alpha_j))(\alpha_{j+1}\cdots\alpha_n)
\subseteq (\im(\alpha_j))(\alpha_{j+1}\cdots\alpha_n)$ and so 
$|\im(\alpha)|\le |(\im(\alpha_j))(\alpha_{j+1}\cdots\alpha_n)| \le |\im(\alpha_j)|$), 
for all $j\in\{1,\ldots,n\}$, and so $|\im(\alpha)|\le|\im(\alpha_i)|<|X|$, as required. 
\end{note}

This last note can be rewritten as follows: 

\begin{lemma}\label{x0}
Let $X$ be an infinite chain such that $X^0=\emptyset$ and let $\alpha\in\O(X)$ be such that $X^+\alpha\cap X^-\ne\emptyset$ or $X^-\alpha\cap X^+\ne\emptyset$. Then $\alpha\not\in\langle J\rangle$. 
\end{lemma}

Before presenting our next note, we introduce the following (natural) notation. 
For $x\in X$ and $Y\subseteq X$, by $x<Y$ (respectively, $x>Y$) we mean that $x<y$ 
(respectively, $x>y$), for all $y\in Y$. 

\begin{note}\label{n3}\em 
Let $X$ be an infinite chain and let $\alpha\in\O(X)$. 
\begin{enumerate}
\item\em If $b\in\im(\alpha)$ and there exists no element $c\in X$ such that $c<b\alpha^{-1}$ then $\im(\alpha)\subseteq[b)$. \em 

In fact, let $y\in\im(\alpha)$. Take $x\in y\alpha^{-1}$. Then $x\not<b\alpha^{-1}$ and so there exists $a\in b\alpha^{-1}$ such that $a\le x$. 
It follows that $b=a\alpha\le x\alpha=y$, whence $y\in[b)$, as required. 

\item\em If $\alpha\in J$ and $b\in\im(\alpha)\cap X^+$ then there exists an element $c\in X$ such that $c<b\alpha^{-1}$. \em 

In fact, if there exists no element $c\in X$ such that $c<b\alpha^{-1}$ then, by 1 above, we have $\im(\alpha)\subseteq[b)$ and, as $b\in X^+$, it follows $|\im(\alpha)|\le|[b)|<|X|$, whence $\alpha\not\in J$, a contradiction. 

\item\em  If $\alpha\in J$ and $y\in X^+$ then there exists an element $b\in\im(\alpha)$ such that $b<y$. \em

In fact, if $y\le b$, for all $b\in\im(\alpha)$, then $\im(\alpha)\subseteq[y)$ and, as $y\in X^+$,  
it follows $|\im(\alpha)|\le|[y)|<|X|$, whence $\alpha\not\in J$, a contradiction. 
\end{enumerate}
\end{note}

By combining 2 and 3 of the previous note, it follows immediately: 

\begin{note} \label{n4}
Let $X$ be an infinite chain such that $X=X^+$, let  $\alpha,\beta\in J$ and let $b\in\im(\alpha)$. Then there exist $c\in X$ and 
$b'\in\im(\beta)$ such that $b'<c<b\alpha^{-1}$. 
\end{note} 

From 3 of Note \ref{n3}, if $X=X^+$, it is clear that $\im(\alpha)$ has no lower bounds, for all $\alpha\in J$. Moreover, we have:

\begin{lemma}\label{x+}
Let $X$ be an infinite chain such that $X=X^+$ (respectively, $X=X^-$) and let $\alpha\in\langle J\rangle$. Then $\im(\alpha)$ has no minimum (respectively, maximum). In particular $J_f\cap\langle J\rangle=\emptyset$. 
\end{lemma}  
\begin{proof} We prove this result for $X=X^+$. The case $X=X^-$ is dual. 

By contradiction, let us suppose that $\im(\alpha)$ has minimum. Denote $\mathrm{min}\im(\alpha)$ by $b_n$. 

As $\alpha\in\langle J\rangle$, we have $\alpha=\alpha_1\alpha_2\cdots\alpha_n$, for some $\alpha_1,\alpha_2,\ldots,\alpha_n\in J$. 

Notice that, since $b_n\in\im(\alpha)$, we also have $b_n\in\im(\alpha_n)$. By applying Note \ref{n4}, we find elements $c_n\in X$ and 
$b_{n-1}\in\im(\alpha_{n-1})$ such that 
$$
b_{n-1}<c_n<b_n\alpha_n^{-1}.
$$
By applying again Note \ref{n4}, we can take elements $c_{n-1}\in X$ and 
$b_{n-2}\in\im(\alpha_{n-2})$ such that 
$$
b_{n-2}<c_{n-1}<b_{n-1}\alpha_{n-1}^{-1}.
$$
Moreover, by Note \ref{n4}, we may recursively construct two 
sequences 
$$
\mbox{$c_n,c_{n-1},\ldots,c_2 \quad\text{and}\quad b_{n-1},b_{n-2},\ldots,b_1$} 
$$
of elements of $X$ such that $b_{i-1}\in\im(\alpha_{i-1})$ and  
$$
b_{i-1}<c_i<b_i\alpha_i^{-1}, 
$$
 for $i=2,\ldots,n$. In addition, by Note \ref{n3}, we may also consider an element $c_1\in X$ such that $c_1<b_1\alpha_1^{-1}$. 
 
Let $i=1,2,\ldots,n$. Then,  $c_i\alpha_i <b_i$. In fact, since $c_i<b_i\alpha_i^{-1}$, we get $c_i\not\in b_i\alpha_i^{-1}$, whence $c_i\alpha_i\ne b_i$, and, given $a\in b_i\alpha_i^{-1}$, we have $c_i<a$ and so $c_i\alpha_i\le a\alpha_i=b_i$. 

Next, by induction on $i$, we prove that $c_1\alpha_1\alpha_2\cdots\alpha_i<b_i$, for $i=1,2,\ldots,n$. 
Let $i=1$. Then, the inequality $c_1\alpha_1<b_1$ was already proved above. 
Hence, let $i>1$ and suppose that $c_1\alpha_1\alpha_2\cdots\alpha_{i-1}<b_{i-1}$, by induction hypothesis. 
Since $b_{i-1}<c_i$, we have $c_1\alpha_1\alpha_2\cdots\alpha_{i-1}<c_i$ and so 
$c_1\alpha_1\alpha_2\cdots\alpha_{i-1}\alpha_i\le c_i\alpha_i<b_i$, as required. 

Hence, in particular, we have $c_1\alpha=c_1\alpha_1\alpha_2\cdots\alpha_n<b_n=\mathrm{min}\im(\alpha)$, which is a contradiction. 
Therefore, $\im(\alpha)$ has no minimum, as required. 
\end{proof}

Next, we state our fundamental lemma. 

\begin{mainlemma}
Let $X$ be an infinite chain. Then  $J_f\subseteq\langle J\rangle$ if and only if $X^0\neq\emptyset$.
\end{mainlemma}   
\begin{proof}
First, suppose that $X^0=\emptyset$. If $X^+=\emptyset$ or $X^-=\emptyset$ then, by Lemma \ref{x+}, we have $J_f\cap\langle J\rangle=\emptyset$, whence $J_f\not\subseteq\langle J\rangle$ (notice that $J_f\ne\emptyset$). 
On the other hand, admit that $X^+\ne\emptyset$ and $X^-\ne\emptyset$. Fix $a\in X^-$ and let $\alpha\in\O(X)$ be the constant transformation with image $\{a\}$. Then $X^+\alpha\cap X^-\ne\emptyset$ and so, by Lemma \ref{x0}, $\alpha\not\in\langle J\rangle$. Since $\alpha\in J_f$, in this case, 
we also obtain $J_f\not\subseteq\langle J\rangle$. 

\smallskip 

Conversely, suppose that $X^0\neq\emptyset$ and fix an element $0\in X^0$. Let $\alpha\in J_f$. 

Suppose, without loss of generality, that $0\alpha \leq 0$ (the case $0\alpha\ge 0$ can be treated dually). 

We begin by defining a transformation $\beta\in\O(X)$ by
$$
x\beta = \left\{\begin{array}{ll}
             x\alpha, & x\leq0 \\
             x, &  x>0.
           \end{array}\right.
$$
Next, let $\im(\alpha) = \{a_1 < a_2 <  \cdots < a_n\}$, $n\in\N$, 
and suppose that $0\alpha = a_i < a_{i+1} < \cdots < a_{i+k} \leq 0$, with  $a_{i+k+1} > 0$ 
or $i + k = n$, for some $i \in \{1, \ldots, n\}$ and a non-negative integer $k$. Then, for $0\leq j \leq k$, 
we define transformations $\gamma_1^{(j)}$ and $\gamma_2^{(j)}$ of $\O(X)$ by
$$
x\gamma_1^{(j)} = \left\{\begin{array}{ll}
             x, &  x < 0 \\
             0, &  x \geq 0 \mbox{ and } x \ngtr a_{i+j}\alpha^{-1} \\
             x, & x > a_{i+j}\alpha^{-1}
           \end{array}\right.
\!\!\!\text{and}~
x\gamma_2^{(j)} = \left\{\begin{array}{ll}
             x, & x < a_{i+j} \\
             a_{i+j}, &  a_{i+j} \leq x \leq 0 \\
             x, &  x > 0.
           \end{array}\right.
$$
Finally, we define a transformation $\delta\in\O(X)$ as being the identity map on $X$ if $i+k = n$ and by
$$
x\delta = \left\{\begin{array}{ll}
             x, &  x \leq 0 \\
             a_{i+k+1}, &  x > 0 \mbox{ and } x \ngtr a_{i+k+1}\alpha^{-1} \\
             x\alpha, &  x > a_{i+k+1}\alpha^{-1}
           \end{array}\right.
$$
otherwise. 

Since $0\in X^0$, it is clear that $\beta, \gamma_1^{(0)}, \gamma_2^{(0)}, \ldots, \gamma_1^{(k)}, \gamma_2^{(k)},\delta \in J$. 
Moreover, we have $\alpha = \beta\gamma_1^{(0)}\gamma_2^{(0)}\cdots\gamma_1^{(k)}\gamma_2^{(k)}\delta$. 
In fact, taking $x \in X$, we may consider three cases:
\begin{enumerate}
\item  If $x \leq 0$ then
$$
(x)\beta\gamma_1^{(0)}\gamma_2^{(0)}\cdots\gamma_1^{(k)}\gamma_2^{(k)}\delta =
(x\alpha)\gamma_1^{(0)}\gamma_2^{(0)}\cdots\gamma_1^{(k)}\gamma_2^{(k)}\delta =
x\alpha,
$$
since $x\alpha \leq a_i$; 

\item  If $x > 0$ and $x\alpha \leq 0$ then there exists $j \in \{0,1,\ldots,k\}$ such that $x\alpha = a_{i+j}$ and
$$
\begin{array}{rcl}
(x)\beta\gamma_1^{(0)}\gamma_2^{(0)}\cdots\gamma_1^{(k)}\gamma_2^{(k)}\delta 
&=& (x)\gamma_1^{(j)}\gamma_2^{(j)}\cdots\gamma_1^{(k)}\gamma_2^{(k)}\delta \\
&=& (0) \gamma_2^{(j)}\cdots\gamma_1^{(k)}\gamma_2^{(k)}\delta \\
&=& (a_{i+j})\gamma_1^{(j+1)}\gamma_2^{(j+1)}\cdots\gamma_1^{(k)}\gamma_2^{(k)}\delta \\
&=& a_{i+j}\\
&=& x\alpha; 
\end{array}
$$

\item  If $x > 0$ and $x\alpha > 0$ then
$$
(x)\beta\gamma_1^{(0)}\gamma_2^{(0)}\cdots\gamma_1^{(k)}\gamma_2^{(k)}\delta = x\delta = x\alpha.
$$
\end{enumerate}

Thus $\alpha\in \langle J\rangle$ and so $J_f\subseteq\langle J\rangle$, as required. 
\end{proof}

The following observation will be useful in the proof of our next result. 

\begin{note}\label{ex}
Let $X$ be an infinite chain. Then $J_f$ contains elements of arbitrary finite (non null) rank. In fact, for all $n\in\N$ and $x_1,x_2,\ldots,x_n\in X$, 
with $x_1<x_2<\cdots<x_n$, we may construct transformations $\alpha\in\O(X)$ such that $\im(\alpha)=\{x_1,x_2,\ldots,x_n\}$. For instance, 
the transformation $\alpha$ on $X$ defined by
$$
x\alpha=\left\{
\begin{array}{ll}
x_1, & x\le x_1\\
x_i, & x_{i-1}<x\le x_i \mbox{ and } 2\le i\le n-1\\
x_n, & x_{n-1}<x
\end{array}
\right. 
$$
belongs to $J_f$. 
\end{note}

Notice that if $X$ is an infinite countable chain then $J_f=\O(X) \setminus J$. 
Thus, in this case, by the previous lemma, we obtain that $\O(X) = \langle J\rangle$ if and only if $X^0 \neq \emptyset$. Furthermore, we have: 

\begin{theorem}\label{mainth}
Let $X$ be an infinite countable chain. The following properties are equivalent:
\begin{enumerate}
\item $\O(X) = \langle J\rangle$, i.e. $\rank(\O(X):J)=0$; 
\item $\rank(\O(X):J)<\aleph_0$; 
\item $X^0 \neq \emptyset$.
\end{enumerate} 
\end{theorem}

\begin{proof}
Notice that 1 trivially implies 2 and, by the previous lemma, 3 implies 1, whence it remains to prove that 2 implies 3.   
Thus, suppose that  $X^0=\emptyset$. Let $\mathcal{C}$ be a generating set of $\O(X)$. 

First, we admit that $X^-\ne\emptyset$ and $X^+\ne\emptyset$. 
As $X^0=\emptyset$, we must have $|X^-|=\aleph_0$ or $|X^+|=\aleph_0$. 
Suppose, without loss of generality, that $|X^-|=\aleph_0$ (the case $|X^+|=\aleph_0$ can be treated dually). 
Hence, given $n\in\N$, we may consider $n$ elements  $x_1,x_2,\ldots,x_n\in X^-$, 
with $x_1<x_2<\cdots<x_n$, and the transformation $\alpha\in\O(X)$ such that $\im(\alpha)=\{x_1,x_2,\ldots,x_n\}$ constructed in Note \ref{ex}. 
Then, for any $x^+\in X^+$, we have $x^+\alpha=x_n\in X^-$ and so, accordingly with Note \ref{n2}, $\mathcal{C}$ contains a transformation $\beta\in\O(X)$ such that $n=|\im(\alpha)|\le|\im(\beta)|<\aleph_0$. Thus, as $n\in\N$ is arbitrary, $\mathcal{C}$ must contain an infinite number of elements of $J_f$. 

On the other hand, admit that $X^+=\emptyset$ or $X^-=\emptyset$. Then, by Lemma \ref{x+}, we have $J_f\cap\langle J\rangle=\emptyset$. 
Let $n\in\N$, let $x_1,x_2,\ldots,x_n\in X$ be 
such $x_1<x_2<\cdots<x_n$ and consider the transformation $\alpha\in\O(X)$ such that $\im(\alpha)=\{x_1,x_2,\ldots,x_n\}$ constructed in Note \ref{ex}. Let $\alpha_1,\ldots,\alpha_k\in\mathcal{C}$ ($k\in\N$) be such that $\alpha=\alpha_1\cdots\alpha_k$.  
Since $J_f\cap\langle J\rangle=\emptyset$ and $\alpha\in J_f$, we must have $\alpha_i\not\in J$, for some $i\in\{1,\ldots,k\}$. 
Hence, $n=|\im(\alpha)|\le|\im(\alpha_i)|<\aleph_0$ (check the proof of Note \ref{n2}). 
Thus, as $n\in\N$ is arbitrary, also in this case, $\mathcal{C}$ must contain an infinite number of elements of $J_f$. 

Therefore, $\rank(\O(X):J)\ge\aleph_0$, as required. 
\end{proof}

Recall that, for $X\in \{\Z,\Q\}$, with the usual order, we have $X^0=X$. Therefore, 
as an immediate consequence of the last theorem, we obtain:

\begin{corollary}
Let $X\in \{\Z,\Q\}$, with the usual order. Then $\O(X)=\langle J\rangle$. 
\end{corollary}

Notice that, for the chain $X= \overset{\mbox{\scalebox{0.4}{$\rightarrow\!\!\cdot\!\!\leftarrow$}}}{\Z}$ defined in the beginning of this section, we have $X^0=\{0\}$, whence also in this case $\O(X)=\langle J\rangle$. 

On the contrary, we have: 
\begin{proposition}
With the usual order of $\N$, we have $\rank(\O(\N):J)=\aleph_0$.
\end{proposition}
\begin{proof} We already observed that $X^0=\emptyset$, for $X=\N$ equipped with the usual order. 
Hence, by Theorem \ref{mainth}, we obtain $\rank(\O(\N):J)\ge\aleph_0$. 
On the other hand, since $J_f=\O(\N)\setminus J$, we have  $\rank(\O(\N):J)\le |J_f|$. 
Therefore, this result follows by showing that $|J_f|=\aleph_0$. In fact, for each $n\in\N$ and each fixed subset $\{x_1,\ldots,x_n\}$ of $\N$ with $n$ elements, we have a bijection between the set $\{\alpha\in\O(\N)\mid\im(\alpha)=\{x_1,\ldots,x_n\}\}$ and the set $\mathscr{P}_{n-1}(\N\setminus\{1\})$ of all subsets of $\N\setminus\{1\}$ with $n-1$ elements, namely  
$\alpha\longmapsto\{\mathrm{min}\,x_i\alpha^{-1}\mid i=2,\ldots,n\}$. Thus, 
since the set $\mathscr{P}_f(\N)$ of all finite subsets of $\N$ has cardinal $\aleph_0$, 
 $J_f$ is an infinite countable union of infinite countable sets and so $|J_f|=\aleph_0$, as required.  
\end{proof}

Observe that our Main Lemma gives us a necessary condition for having $\O(X)=\langle J\rangle$, namely $X^0\ne\emptyset$. 
We finish this note by presenting a sufficient condition: 

\begin{theorem} 
Let $X$  be an infinite chain such that $X\setminus X^0$ is finite. Then $\O(X)=\langle J\rangle$.
\end{theorem}
\begin{proof}
Notice that, $X^-$ and $X^+$ are both finite sets and $|X^0|=|X|$.

Take $\alpha\in\O(X)$. 

First, suppose that $X^0\alpha\cap X^0\ne\emptyset$. Fix $u,v\in X^0$ such that $u\alpha=v$.
If $u\le v$, we define transformations $\alpha_1$ and $\alpha_2$ of $\O(X)$ by
$$
x\alpha_1
=
\left\{
\begin{array}{ll}
x, &  x<u\\
x\alpha, & u\le x
\end{array}
\right.
\quad\text{and}\quad
x\alpha_2
=
\left\{
\begin{array}{ll}
x\alpha, & x< u\\
v, & u\le x < v\\
x, &  v\le x.
\end{array}
\right.
$$
On the other hand, if $v<u$, we define transformations $\alpha_1$ and $\alpha_2$ of $\O(X)$ by
$$
x\alpha_1
=
\left\{
\begin{array}{ll}
x\alpha, &  x\le u\\
x, &  u < x
\end{array}
\right.
\quad\text{and}\quad
x\alpha_2
=
\left\{
\begin{array}{ll}
x, & x< v\\
v, & v\le x < u\\
x\alpha, &  u\le x.
\end{array}
\right.
$$

It is a routine matter to show that both cases satisfy $\alpha_1,\alpha_2\in J$ and $\alpha=\alpha_1\alpha_2$. Then $\alpha\in\langle J\rangle$. 

On the other hand,  suppose that $X^0\alpha\cap X^0=\emptyset$. Then $\alpha\in J_f$ and so, 
by our Main Lemma, we obtain again $\alpha\in\langle J\rangle$, as required.  
\end{proof}

Clearly, the converse of this property is not valid in general, 
as the example $X=\overset{\mbox{\scalebox{0.4}{$\rightarrow\!\!\cdot\!\!\leftarrow$}}}{\Z}$ shows. 
Nevertheless, as an immediate application, for the usual chain of real numbers, we have:

\begin{corollary}
With the usual order of $\R$, we have $\O(\R)=\langle J\rangle$. 
\end{corollary}

\lastpage 

\end{document}